\documentclass[reqno,11pt]{amsart}
\usepackage{amsmath,amsfonts,amssymb,amsthm}
\usepackage{mathtools,mathrsfs,yhmath}
\usepackage{graphicx,color,xcolor}
\usepackage{enumitem}
\usepackage{hyperref}
\usepackage[style=alphabetic, backend=biber, sorting=nty]{biblatex}

\bibliography{lift_rw_bibliography.bib}

 \allowdisplaybreaks

\DeclareMathOperator{\grad}{\mathrm{grad}}
\DeclareMathOperator{\Hess}{\mathrm{Hess}}

\DeclareMathOperator{\Supp}{\mathrm{Supp}}
\DeclareMathOperator{\Id}{\mathrm{Id}}

\newcommand{\E}{\mathbb{E}}

\renewcommand{\L}{\mathcal{L}}
\renewcommand{\H}{\mathcal{H}}

\newcommand{\R}{\mathbb{R}}
\renewcommand{\S}{\mathcal{S}}
\newcommand{\V}{\mathcal{V}}
\newcommand{\norm}[1]{\left\lVert#1\right\rVert}

\numberwithin{equation}{section}
\newtheorem{theorem}{Theorem}[section]

\newtheorem{proposition}[theorem]{Proposition}
\newtheorem{corollary}[theorem]{Corollary}

\newtheorem{assumption}[theorem]{Assumption}
\newtheorem{definition}[theorem]{Definition}

\newtheorem{remark}[theorem]{Remark}

\begin{document}

\title{Invariance principle for Lifts of Geodesic Random Walks}
\author{Jonathan Junn\'e} \author{Frank Redig} \author{Rik Versendaal}
\thanks{TU Delft, Delft Institute of Applied Mathematics, Mekelweg 4, 2628 CD Delft, Netherlands.\newline Emails: \textsf{j.junne@tudelft.nl} \textsf{f.h.j.redig@tudelft.nl} \textsf{r.versendaal@tudelft.nl}}

\begin{abstract}
    We consider a certain class of Riemannian submersions $\pi : N \to M$ and study lifted geodesic random walks from the base manifold $M$ to the total manifold $N$. Under appropriate conditions on the distribution of the speed of the geodesic random walks, we prove an invariance principle; i.e.,  convergence to horizontal Brownian motion for the lifted walks.  This gives us  a natural probabilistic proof of the geometric identity relating the horizontal Laplacian $\Delta_\H$ on $N$ and the Laplace-Beltrami operator $\Delta_M$ on $M$. In particular, when $N$ is the orthonormal frame bundle $O(M)$, this identity is central in the Malliavin-Eells-Elworthy construction of Riemannian Brownian motion.
\end{abstract}
\keywords{ Invariance principle, geodesic random walks, horizontal Laplacian, Riemannian Brownian motion, Riemannian submersions}
\maketitle

\section{Introduction}
In this paper we consider geodesic random walks on a Riemannian manifold $(M,g)$ and consider their horizontal lift into
a manifold $(N,\tilde{g})$ such that there is a Riemannian submersion $\pi: N\to M$.
A motivating example of this setting is the orthonormal frame bundle $\pi_{O(M)} : O(M) \to M$ of a Riemannian manifold.
This example is the basis of the Malliavin-Eells-Elworthy construction of Brownian motion. The important point in this setting is that the horizontal Brownian motion has as a generator the horizontal Laplacian which is a sum of squares of globally defined vector fields; i.e., it is in \emph{H\"{o}rmander form}
\begin{equation*}
    \Delta_{\H}= \sum_{i=1}^d H_i^2,
\end{equation*}
where $d$ is the dimension of the manifold.
Because of this, the Markov process generated by $\Delta_{\H}$ can be constructed as the solution of a Stratonovich SDE \cite{Hsu02} driven by an
$\R^d$-valued Brownian motion.
Then, the Brownian on the manifold is the projection of this horizontal Brownian motion. This is based on the fact that
\begin{equation}\label{horpro}
\Delta_{\H} (f\circ \pi) = (\Delta_M f) \circ \pi
\end{equation}
for all smooth $f:M\to\R$. The proof of identity \eqref{horpro} in \cite{Hsu02} is based on an explicit somewhat involved computation.
Horizontal Brownian motion is extensively studied in Baudoin's monograph \cite{BDW22}.

Brownian motion on $M$ can  be obtained as a scaling limit of geodesic random walks as initially considered by J\o rgensen \cite{Jor75}. It is therefore natural to lift these walks horizontally in order to obtain horizontal Brownian motion in the scaling limit. As a consequence of such a weak converge result, the horizontal Brownian motion on the total manifold $N$ and the Brownian motion on the base manifold $M$ are then $\pi$-related automatically.
It is precisely the aim of our paper to prove this result for a class of geodesic random walks, in the setting of Riemannian submersions.
We start in section 3 by proving this invariance principle, and its corollary \eqref{horpro} for the orthonormal frame bundle; i.e., the context of \cite{Hsu02}. Then in section 4 we consider general submersions where we prove the same result, and provide several examples.

\section{Random walks and horizontal random walks}
In this section we introduce the stochastic processes we study, namely, horizontal random walks. To do so, we first introduce the analogue of random walks in $M$, so-called geodesic random walks, following \cite{Jor75,Ver19}. Afterwards, we explain how these geodesic random walks can be lifted to the total space $N$ along a Riemannian submersion $\pi:N \to M$.
\subsection{Geodesic random walk}
We consider a $d$-dimensional geodesically complete Riemannian manifold $M$ with metric $g$, and denote by $T_pM$ the tangent space of $M$ at $p\in M$.
In order to describe increments of our random walks, we have to consider a collection of
probability measures $\mu_p$ on $T_p M$.
We then say that $\mu_p$ is \emph{measurable} (or \emph{continuous}, \emph{smooth}) as a function of $p$ if for every smooth coordinate system $x$ about $p$, the associated family of measures
on $\R^d$  (where we use $\{\partial/\partial x^i\}_{1\le i\le d}$ as a basis for $T_pM$ to identify $T_pM$ with $\R^d$) is measurable (or continuous, smooth).
By smoothness of the transition maps, if this holds for one smooth coordinate system, then it holds for all smooth coordinate systems.

Such a collection of probability measures $\mu_p$ on $T_p M$, $p\in M$, depending in a measurable way on $p\in M$, is called a distribution of \emph{increments}.

The nomenclature \emph{increment} is inspired from \cite{Jor75} where $\mu_p$ describes the direction in which the random walk follows
a geodesic when starting from $p$.
More precisely, we define the following Markov processes based on $\{\mu_p\}_{p \in M}$:
\begin{definition}
\begin{enumerate}
\item The \emph{discrete-time unit speed random walk} $\{S_k\}_{ k\in\mathbb{N}}$ based on $\{\mu_p\}_{p \in M}$ is defined via its transition operator
\begin{equation}\label{transop}
Pf(p) := \E \left[f\left(S_{k+1}\right)|S_k=p\right]= \int_{T_pM} f(\exp_p(v)) \: \mu_p(dv);
\end{equation}
\item The \emph{discrete-time random walk with speed $\alpha$} based on $\{\mu_p\}_{p \in M}$
is denoted by $\{\S^{(\alpha)}_k\}_{ k\in\mathbb{N}}$ and
is defined via its transition operator
\begin{equation}\label{transopa}
P^{(\alpha)}f(p) := \E \left[f\left(S^{(\alpha)}_{k+1}\right)|S^{(\alpha)}_k=p\right] = \int_{T_pM} f(\exp_p(\alpha v)) \:\mu_p(dv);
\end{equation}
\item Finally, the continuous time process $\{Z^{(\alpha)}\}_{t\geq 0}$ is defined via its generator.
\begin{equation}\label{genalpha}
L^{(\alpha)} f(p) := \alpha^{-2}\left(P^{(\alpha)}f(p) - f(p)\right)
\end{equation}
\end{enumerate}
\end{definition}
The process $\{S_k\}_{k\in\mathbb{N}}$ evolves as follows: whenever $S_k=p$, $S_{k+1}$ is obtained
by randomly choosing $X_{k+1}$ on $T_pM$ according to the measure $\mu_p$ and following the geodesic starting at $p$ in
the direction $X_{k+1}$ for time $1$, and analogously for the walk with speed scaled by $\alpha$.

In what follows, we want to  prove weak convergence to (horizontal) Brownian motion for the continuous walk $\{Z^{(\alpha)}_t\}_{t\geq 0}$ and its horizontal lift (defined below)
as $\alpha$ tends to zero.
This will then imply immediately the same weak convergence 
results for the discrete walk $\{S^{(\alpha)}_{\lfloor\alpha^{-2}t\rfloor}\}_{t\ge 0}$ as $\alpha$ tends to zero.

In order to proceed, we need some conditions on the distribution of increments.
Because we aim at proving convergence to Brownian motion, there is a centering and variance condition.
Finally, in order to prove uniform convergence of generators, it is convenient to have an additional third moment condition.
More precisely, we make the following assumptions:
\begin{assumption}[Centering and covariance]\label{assumption:unit_variance}
    For every $p \in M$, the measure $\mu_p$ has zero expectation and its  covariance equals the inverse metric; i.e.,
    \begin{equation*}
        \int_{T_pM} v \: \mu_p(dv) = 0, \quad \int_{T_pM} v \otimes v \: \mu_p(dv) = g^{-1}(p),
    \end{equation*}
    or equivalently, in any smooth coordinates system about $p$,
    \begin{equation}\label{equation:expectation_variance}
        \int_{T_pM} v^i \: \mu_p(dv) = 0, \quad \int_{T_pM} v^iv^j \:\mu_p(dv) = g^{ij}(p), \quad i,j = 1, \dots, d.
    \end{equation}
\end{assumption}
\begin{assumption}[Third moment condition]\label{assumption:third_moment}
    The third moment of the collection of measures $\{\mu_p\}_{p\in M}$ is finite, uniformly on compacts; i.e., for
    all $K\subset\subset M$ compact,
    \begin{equation*}\label{equation:third_moment}
        \sup_{p\in K}\int_{T_pM} \norm{v}^3 \: \mu_p(dv) < +\infty, \quad
    \end{equation*}
\end{assumption}
\begin{remark}
If $p\mapsto \mu_p$ is invariant under parallel transport, which is the condition considered in \cite{Jor75} in order to mimic identically distributed increments, then if assumptions \ref{assumption:unit_variance} and \ref{assumption:third_moment} are satisfied for a single point $p\in M$, then they are for all points $p\in M$.
\end{remark}

\subsection{Horizontal lift of geodesic random walks}
Now that we have defined geodesic random walks on the base manifold $(M, g)$, we can construct a new process on the total manifold $(N, \tilde g)$ carrying a metric $\tilde g$ that will be specified later on. In order to define this process, we recall some terminology from Riemannian geometry.

\begin{definition}
    A \emph{Riemannian submersion} $\pi : (N,\tilde g) \to (M, g)$ is a smooth surjective map whose differential is an isomorphism
    \begin{equation*}
        d\pi_u : \left(\ker d\pi_u\right)^\perp \to T_{\pi(u)}M
    \end{equation*}
    which is also an isometry.  Here $\perp$ denotes the orthogonal complement with respect to the metric $\tilde g$ in $N$.
\end{definition}
In the setting of Riemannian submersions, 
the tangent space $T_uN$ of the total manifold $N$ at a point $u\in N$ splits into the \emph{vertical} and \emph{horizontal subspaces} as follows:
    \begin{equation*}
        \V_u N := \ker d\pi_u, \quad \H_u N := (\V_u N)^\perp, \quad T_u N = \H_u N \oplus \V_u N.
    \end{equation*}
Their disjoint unions form two subbundles of $TN$ denoted respectively by $\V N = \sqcup_{u\in N} \V_u N$ and $\H N = \sqcup_{u\in N} \H_uN$. This splits the metric $\tilde g$ on $TN$ into its two factors $g_{\V N}$ and $g_{\H N}$. A horizontal vector field $X \in \Gamma(\H N)$ is \emph{$\pi$-related} to a vector field $\overline{X} \in \Gamma(TM)$ if for any $u \in N$ it holds
    \begin{equation}\label{related}
        d\pi_u\left(X_u\right) = \overline{X}_{\pi(u)}.
    \end{equation}
We stress out that relating manifolds via a Riemannian submersion make sure that horizontal $\pi$-related tangent vectors 
as in \eqref{related} have the same norm because $d\pi_{|_{\H N}}$ is an isometry. 

\begin{remark}
In several situations, the total manifold $N$ comes with a natural projection map $\pi : N \to M$ defining the vertical subspaces $\V_u N = \ker d\pi_u$ but no specification of a metric $\tilde g$. One can then use any connection form $\omega$ to define the horizontal subspaces $\H_u N = \ker \omega_u$. Now, with the help of this choice of horizontal bundle, obtained either by the Riemannian submersion or by the specification of a connection form, one can lift any smooth curve on the base manifold to the total manifold with respect to the horizontal bundle.
\end{remark}

We can now define the horizontal lift of a curve $\gamma: I\to M$. We denote $\gamma'(t)=\tfrac{d}{dt}(\gamma(t))$.

\begin{definition}
    The \emph{horizontal lift} $\tilde{\gamma}$ with respect to $\H N$ starting at $u_0 \in N_{\gamma(0)} = \pi^{-1}(\{\gamma(0)\})$ of a smooth curve $\gamma : I \to M$ is the unique curve satisfying
    \begin{equation}\label{equation:horizontal_lift}
        \pi \circ \tilde{\gamma} = \gamma, \quad \tilde{\gamma}'(t) \in \H_{\tilde{\gamma}(t)}N.
    \end{equation}
\end{definition}
    Similarly, the \emph{horizontal lift} $\tilde v$ with respect to $\H N$ starting at $u$ of a tangent vector $v \in T_{\pi(u)}M$ is given by differentiating (\ref{equation:horizontal_lift});
    \begin{equation*}
        \tilde v = (d\pi_u)^{-1}v \in \H_uN.
    \end{equation*}
If $\tilde{v}$ is the horizontal lift of $v \in T_p M$, then for every $\gamma$ in $M$ such that $\gamma(0)=p$, and
$\gamma'(0)=v$, and for every $u\in N$ such that $\pi(u)=p$, the horizontal lift of $v$, denoted by $\tilde{v}= \tilde{v}[v,u]$ equals $\tilde{\gamma}'(0)$, where $\tilde{\gamma}$ is the horizontal lift of $\gamma$
starting at $\tilde{\gamma}(0)=u$.

We recall that the horizontal lift of a geodesic under a Riemannian submersion is again a geodesic (see for instance \cite[Lemma 26.11]{Mic08}).
It is important to notice that geodesics in $N$ with initial horizontal tangent vector, are horizontal curves; i.e.,
the tangent vectors remain horizontal.
Moreover, by the geodesic property, the tangent vector at any point of the curve is the parallel transport of the initial tangent vector.
\begin{definition}
    Given a distribution of increments $\{\mu_p\}_{ p\in M}$ we define its \emph{horizontal lift} $\{\tilde{\mu}_u\}_{ u\in N}$ as follows: The distribution $\tilde{\mu}_u$ is obtained by first drawing $v$ according to $\mu_{\pi(u)}$ and then lifting $v$ to $\tilde{v}[v,u]$.
    It then follows that the (discrete or continuous-time) random walks based on $\{ \mu_p \}_{ p\in M }$ are horizontally lifted to the (discrete or continuous-time) random walks based on $\{ \tilde{\mu}_u \}_{ u\in N}$, and conversely, the projections of random walks based on $\{ \tilde{\mu}_u \}_{u\in N}$ are distributed as the random walks based on $\{\mu_p\}_{p\in M}$.
\end{definition}
As a consequence, the \emph{horizontal lift of the rescaled continuous-time random walk} $\{Z^{(\alpha)}_t\}_{ t\geq 0 }$ defined via its generator  (\ref{genalpha}) is the process $\{\tilde{Z}^{(\alpha)}_t\}_{ t\geq 0}$ on the total manifold $N$ with generator
defined on smooth compactly supported functions $f:N\to\R$ by
\begin{equation}\label{equation:resclaled_generator}
        \L_\alpha f(u) := \alpha^{-2} \int_{T_{\pi(u)}M} f\left(\exp_u (\alpha\tilde{v}[v,u])\right) - f(u)\: \mu_{\pi(u)}(dv), 
\end{equation}
    
so that $\{\pi(\tilde{Z}^{(\alpha)}_t)\}_{ t\geq 0 }= \{Z^{(\alpha)}_t\}_{ t\geq 0}$ in distribution.

\section{Invariance principle for the orthonormal frame bundle}
We now turn to our main result, namely the invariance principle for horizontal random walks.
Before we state this result in full generality in Section 4, we first consider the special case in which $N$
is the orthonormal frame bundle. The reason for this is two-fold. First of all, the (orthonormal) frame
bundle plays a central role in defining stochastic processes in manifolds by constructing them from their
Euclidean counterparts. This motivated our study of horizontal random walks. Second, considering
the orthonormal frame bundle allows for a more streamlined proof of the invariance principle, therefore
making it more instructive to consider first.
Before we can state the main theorem, we need additional definitions. We start by defining the orthonormal frame bundle.

\begin{definition}
    An \emph{orthonormal frame} $u$ at $p$ is an ordered choice of orthonormal basis $\{ue_i\}_{1\le i\le d}$ of $T_pM$, where $\{e_i\}_{1\le i\le d}$ is the canonical basis of $\R^d$. The set of all orthonormal frames at $p$ is denoted $O_p$ and their disjoint union $O(M) := \sqcup_{p \in M} O_p$ is referred to as the \emph{orthonormal frame bundle}.
\end{definition}
The orthonormal frame bundle $O(M)$ is a manifold of dimension $d(d+1)/2$ that comes with a natural submersion $\pi_{O(M)} : O(M) \to M$ sending any orthonormal frame $u \in O_p$ to the basepoint $p$. If $(U, \{x^i\}_{1\le i\le d})$ is a local chart in $M$ about $p$, we can express the orthonormal basis of $T_pM$ as $ue_i = (ue_i)^j \frac{\partial}{\partial x^j}$, and this gives a local chart $(\pi^{-1}(U), (\{x^k\}_{1\le k\le d}, \{(ue_i)^j\}_{1\le i<j\le d}))$ in $O(M)$ about $u$. It remains to define a splitting of $TO(M)$, for instance, by specifying a notion of horizontality. 
\begin{definition}
    A smooth curve $u : I \to O(M)$ is \emph{horizontal} if for any $e\in \R^d$ the tangent vector field $u(t)e \in T_{\pi(u(t))}M$ is itself parallel with respect to the Levi-Civita connection $\overline{\nabla}$ on $M$ along the curve $\pi \circ u : I \to M$. 
\end{definition}
This notion of horizontality induces the splitting $TO(M) = \H O(M) \oplus \V O(M)$ and allows us to lift smooth curves horizontally. Given a smooth $\gamma : I \to M$ and its horizontal lift $\tilde{\gamma}$ starting at $u$, we recover the parallel transport of tangent vectors along $\gamma$ given by
\begin{equation*}
    \tau_{\gamma; t_1t_2} : T_{\gamma(t_1)} \to T_{\gamma(t_2)} : v \mapsto \tilde{\gamma}(t_2)\left(\tilde{\gamma}(t_1)\right)^{-1}v.
\end{equation*} We can look at the horizontal lifts of the different orthonormal basis $ue_i$ of $T_pM$ induced by the orthonormal frame $u$ at $p$ for each $e_i$.
\begin{definition}
    Let $u$ be an orthonormal frame at $p$. The \emph{canonical horizontal vector fields}
    \begin{equation}\label{equation:fundamental_horizontal}
        H_i(u) := \widetilde{ue_i}, \quad i = 1,\dots,d,
    \end{equation}
    are the horizontal lifts with respect to $\H O(M)$ of the tangent vectors $ue_i \in T_pM$ starting at $u$.
\end{definition}
To find a coordinate expression for these vector fields, consider a horizontal lift $\tilde\gamma : I \to O(M)$ that starts at $u$ with $\gamma'(0) = ue_i$. By definition of horizontal lift with respect to $\H O(M)$,
\begin{equation*}
    H_i(u) = \widetilde{ue_i} = \tilde\gamma'(0) = \dot\gamma^j(0)\frac{\partial}{\partial x^j} +  \left(\left(\tilde{\gamma}e_l\right)^m\right)'(0)\frac{\partial}{\partial(ue_l)^m},
\end{equation*} and since the tangent vectors $\tilde{\gamma}(t)e_l$ are parallel with respect to  $\overline{\nabla}$ on $M$ along the curve $\pi \circ \tilde\gamma : I \to M$ whose initial tangent is vector $ue_i$, the geodesic equation yields
\begin{equation*}
    \left(\left(\tilde{\gamma}e_l\right)^m\right)'(0) = -(ue_i)^j(ue_l)^k\overline{\Gamma}^m_{jk}, \quad 1 \le l < m \le d,
\end{equation*}
where $\{\overline{\Gamma}^k_{ij}\}_{1\le i,j,k\le d}$ denote the Christoffel symbols of $\overline{\nabla}$.
The horizontal and vertical subbundles of $TO(M)$ are thus respectively spanned by (see \cite[Proposition 2.1.3]{Hsu02})
\begin{equation}\label{equation:horizontal}
    H_i(u) = (ue_i)^j \left(\frac{\partial}{\partial x^j} - (ue_l)^k\overline{\Gamma}_{jk}^m\frac{\partial}{\partial (ue_l)^m}\right), \quad i = 1, \dots, d,
\end{equation}
and
\begin{equation}\label{equation:vertical}
    \quad V^k_j(u) := (ue_j)^l\frac{\partial}{\partial (ue_k)^l}, \quad 1 \le j < k \le d.
\end{equation}
A natural choice of metric compatible with this splitting is a Sasaki-Mok type metric introduced in \cite{Sa58} and \cite{Mok78} (see also \cite{KoOl08}).
\begin{definition}
    The \emph{canonical $1$-form} $\theta$ and the \emph{connection form} $\omega$ on $O(M)$ associated to $\overline{\nabla}$ on $M$  are the dual forms to the vector fields (\ref{equation:horizontal}) and (\ref{equation:vertical}) given by
    \begin{equation*}\label{equation:1-form_canonical}
        \theta^k(u) := (ue_l)^k dx^l, \quad \omega^j_i(u) := (ue_k)^j\left(\overline{\Gamma}^k_{lm} (ue_i)^l dx^m + d(ue_i)^k\right).
    \end{equation*}
    The \emph{Sasaki-Mok metric} $\tilde g$ is defined pointwise by
    \begin{equation*}\label{equation:Sasaki-Mok_metric}
        \langle \eta, \xi \rangle_{\tilde g} := \left\langle d\pi_{O(M),u} (\eta) , d\pi_{O(M),u}(\xi) \right\rangle_g + \left\langle \omega_u(\eta), \omega_u(\xi)\right\rangle, 
    \end{equation*}
    where $\langle \cdot, \cdot \rangle$ denotes an $O(d)$-invariant inner product on $\mathfrak{o}(d)$.
\end{definition} 
The global canonical horizontal vector fields $H_i \in \Gamma(\H O(M))$ allow us to define a horizontal Laplacian for the orthonormal frame bundle as a sum of squares:
\begin{definition}
    The \emph{horizontal Laplacian} of $O(M)$ is given by
    \begin{equation}\label{equation:horizontal_laplacian_O(M)}
        \Delta_\H = \sum_{i=1}^d H_i^2.
    \end{equation}
\end{definition}
This operator is in \emph{H\"ormander's form}. In general, Nash's embedding theorem allows one to write the Laplace-Beltrami operator of $M$ as a sum of squares of orthogonal projections (see for instance \cite[Theorem 3.1.4]{Hsu02}) at the cost of extra terms coming from the dimension of the isometric embedding. The horizontal Laplacian and the Laplace-Beltrami operator satisfy the following relation, and this is a starting point in stochastic calculus on manifolds using (anti-)development:
\begin{proposition}\label{theorem:fundamental_relation}
    The following identity holds: \begin{equation}\label{equation:fundamental_relation}
        \Delta_\H \left(f \circ \pi_{O(M)}\right) = \Delta_M f \circ \pi_{O(M)}, \quad f \in C^\infty(M).
    \end{equation}
\end{proposition}

The proofs of Proposition \ref{theorem:fundamental_relation} in \cite[Proposition 3.1.2]{Hsu02} and \cite[Proposition 4.2.7]{BDW22} are more geometric in essence. Here, we deduce this relation as a corollary to the invariance principle of the horizontally lifted geodesic random walks on the orthonormal frame bundle.
\begin{theorem}[Invariance principle on the orthonormal frame bundle]\label{theorem:invariance_frame_bundle}
Let $(M,g)$ be a geodesically complete Riemannian manifold and let $\{\mu_p\}_{p\in M}$ be a distribution of increments on $M$ satisfying Assumption \ref{assumption:unit_variance} and Assumption \ref{assumption:third_moment}. 
Let $\{\tilde{Z}^{(\alpha)}_t, t\geq 0\}$ be the process with generator \eqref{equation:resclaled_generator}.
Then as $\alpha\to 0$, this process converges to horizontal Brownian motion; i.e., the process with generator $\frac12\Delta_\H$.
\end{theorem}
\begin{proof}[Proof of Theorem \ref{theorem:invariance_frame_bundle}]
    Let $f : M \to \R$ be a smooth compactly supported function.
    Given a frame $u$ for $p$, we perform a Taylor's expansion of $f \circ \widetilde{\gamma_\alpha}$, where $\widetilde{\gamma_\alpha}$ is the horizontal lift starting at $u$ of the curve $\gamma_\alpha(t) = \exp_p(\alpha tv)$. There is some $0 < s < 1$ such that
\begin{equation*}
    f \circ \widetilde{\gamma_\alpha}(1) = f(u) + \left.\frac{d}{dt}\right|_{t=0} \left(f \circ \widetilde{\gamma_\alpha}(t)\right) + \frac{1}{2} \left.\frac{d^2}{dt^2}\right|_{t=0} \left(f \circ \widetilde{\gamma_\alpha}(t)\right) + \frac{1}{6} \left.\frac{d^3}{dt^3}\right|_{t=s} \left(f \circ \widetilde{\gamma_\alpha}(t)\right).
\end{equation*}
To compute the time derivatives, first note that the horizontal lift allows us to express
\begin{equation*}
    \gamma_\alpha'(t) = \tau_{\gamma_\alpha; 0t} \left(\alpha v\right) = \widetilde{\gamma_\alpha}(t)\left(\widetilde{\gamma_\alpha}(0)\right)^{-1} (\alpha v) = \alpha\widetilde{\gamma_\alpha}(t) u^{-1}v = \alpha \left(u^{-1} v\right)^i \widetilde{\gamma_\alpha}(t) e_i.
\end{equation*}
Thus
\begin{equation*}
    \frac{d}{dt} \left(f \circ \widetilde{\gamma_\alpha}(t)\right) = df_{\widetilde{\gamma_\alpha}(t)} \circ \widetilde{\gamma_\alpha}'(t) = \alpha\left(u^{-1} v\right)^i H_i f(\widetilde{\gamma_\alpha}(t)),
\end{equation*}
and likewise,
\begin{multline*}
    \frac{d^2}{dt^2} f(\widetilde{\gamma_\alpha}(t)) = \alpha^2\left(u^{-1} v\right)^j \left(u^{-1} v\right)^i H_j H_i f(\widetilde{\gamma_\alpha}(t)), \\
    \frac{d^3}{dt^3} f(\widetilde{\gamma_\alpha}(t)) = \alpha^3\left(u^{-1} v\right)^k\left(u^{-1} v\right)^j \left(u^{-1} v\right)^i H_k H_j H_i f(\widetilde{\gamma_\alpha}(t)).
\end{multline*}
Since $u$ is an orthonormal frame, $\{ue_i\}_{1\le i\le d}$ is an orthonormal basis of $T_pM$, and we get
\begin{equation*}
    v^i = \langle v, ue_i \rangle_g = \langle u^{-1}v, e_i\rangle_{\R^d} = \left(u^{-1} v\right)^i.
\end{equation*}
Recall the rescaled generator (\ref{equation:resclaled_generator}) given by
\begin{multline*}
    \alpha^{-2} \int_{T_pM} f \circ \widetilde\gamma_\alpha(1) - f(u) \: \mu_p(dv) \\ = \alpha^{-2} \int_{T_pM} \left.\frac{d}{dt}\right|_{t=0} \left(f \circ \widetilde{\gamma_\alpha}(t)\right) + \frac{1}{2} \left.\frac{d^2}{dt^2}\right|_{t=0} \left(f \circ \widetilde{\gamma_\alpha}(t)\right) + \frac{1}{6} \left.\frac{d^3}{dt^3}\right|_{t=s} \left(f \circ \widetilde{\gamma_\alpha}(t)\right) \:\mu_p(dv).
\end{multline*}
The first term vanishes by the centering of the collection of measures, and the second one is precisely the horizontal Laplacian. Indeed, under Assumption \ref{assumption:unit_variance}, by the linearity of the integral and the horizontal lift, we get that
\begin{equation*}
    \alpha^{-2} \int_{T_pM} \left.\frac{d}{dt}\right|_{t=0} \left(f \circ \widetilde{\gamma_\alpha}(t)\right) \: \mu_p(dv) =  \alpha^{-1} df_u \circ \left(d\pi_u\right)^{-1} \left( \int_{T_pM} v \:\mu_p(dv)\right) = 0.
\end{equation*}
Moreover, by (\ref{equation:expectation_variance}), we have
\begin{align*}
    \alpha^{-2} \int_{T_pM} \left.\frac{d^2}{dt^2}\right|_{t=0} \left(f \circ \widetilde{\gamma_\alpha}(t)\right) \:\mu_p(dv) &= \int_{T_pM} \left(u^{-1} v\right)^j \left(u^{-1} v\right)^i \:\mu_p(dv) H_j H_i f(u) \\
        &= \int_{T_pM} v^jv^i \:\mu_p(dv) H_j H_i f(u) \\
        &= \sum_{i=1}^d H_i^2 f(u),
\end{align*}
which is the horizontal Laplacian. For the third order term, we conclude using Assumption \ref{assumption:third_moment} as follows: $\pi_{O(M)}\left(\Supp f\right)$ is compact by continuity of the projection, and we estimate
\begin{align*}
    \alpha^{-2}\int_{T_pM} \left.\frac{d^3}{dt^3}\right|_{t=s} \left(f \circ \widetilde{\gamma_\alpha}(t)\right) \:\mu_p(dv) &= \alpha \int_{T_pM} v^kv^jv^i H_k H_j H_i f({\widetilde\gamma_\alpha}(s)) \:\mu_p(dv)  \\
        &\le \alpha\int_{T_pM} \norm{v}^3 \:\mu_p(dv) \sum_{i,j,k=1}^d\norm{H_k H_j H_i f}_\infty \\
        &\le \alpha \sup_{q\in \pi_{O(M)}\left(\Supp f\right)} \int_{T_qM} \norm{v}^3 \:\mu_q(dv) \sum_{i,j,k=1}^d\norm{H_k H_j H_i f}_\infty,
\end{align*}
which goes to $0$ independently of the frame u as $\alpha \to 0$.
\end{proof}

\section{Invariance principle for Riemannian submersions}
In this Section, we extend the results of Section 3 for the orthonormal frame bundle to the more general framework of Riemannian submersions. 

Let $(M, g)$ and $(N, \tilde{g})$ be Riemannian manifolds with Riemannian submersion $\pi : N \to M$, and let $d$ be the dimension of $M$. Each vector field $X \in \Gamma(TN)$ can be decomposed uniquely into its horizontal part $X_\H \in \H$ and vertical part $X_\V \in \V$ respectively. Under this setting, we consider the Laplace-Beltrami operator $\Delta_N$ on $N$, or even its horizontal and vertical parts as follows:
\begin{definition}
    The \emph{horizontal Laplacian} $\Delta_\H$ is the generator of the pre-Dirichlet form
    \begin{equation*}
        \mathcal{E}_\H(f, h) = -\int_E \langle (\grad f)_\H, (\grad h)_\H\rangle_{\tilde{g}} \: d\mathrm{Vol}_{\tilde{g}}, \quad f, h \in C^\infty_c(N).
    \end{equation*}
    In local orthonormal frames $\{E_i\}_{1\le i\le d}$ of $\H$ and $\{F_j\}_{1\le j \le l}$ of $\V$, this operator can be rewritten as
    \begin{equation}\label{equation:full_horizontal_laplacian}
        \Delta_\H = \sum_{i=1}^d \left(E_i^2 - \left(\nabla_{E_i} E_i\right)_\H\right) - \sum_{j=1}^l \left(\nabla_{F_j} F_j\right)_\H.
    \end{equation}
    Analogously, the \emph{vertical laplacian} $\Delta_\V$ is the vertical part of the Laplace-Beltrami operator on $E$, and
    \begin{equation*}
        \Delta_N = \Delta_\H + \Delta_\V.
    \end{equation*}
\end{definition}
Of course, since the $F_j$'s are vertical, they cannot possibly be obtained as horizontal lifts. The last term in the horizontal Laplacian (\ref{equation:full_horizontal_laplacian}) should thus vanish in order to obtain convergence of the generator of horizontally lifted geodesic random walks towards this operator. The following type of Riemannian submersion ensures that the last term indeed vanishes; $ \left(\nabla_{F_j} F_j\right)_\H = 0$ (see \cite{ON66} and \cite[Proposition 4.13]{ON83}).
\begin{definition}
    The fibers $N_p = \pi^{-1}(\{p\})$ of a Riemannian submersion $\pi : N \to M$ are said to be \emph{totally geodesic} if any geodesic in a fiber, seen as a submanifold of $N$ with the induced metric, is also a geodesic in $N$.
\end{definition}
Assuming that the submersion has totally geodesic fibers, the horizontal Laplacian (\ref{equation:full_horizontal_laplacian}) takes the form
\begin{equation}\label{equation:horizontal_laplacian_submersion}
    \Delta_\H = \sum_{i=1}^d \left(E_i^2 - \left(\nabla_{E_i} E_i\right)_\H\right).
\end{equation}
This operator is generally not in \emph{H\"ormander's form} as in the special case (\ref{equation:horizontal_laplacian_O(M)}) of the orthonormal frame bundle which is a parallelizable manifold.
H\"ormander's theorem allows us to check the subellipticity of $\Delta_\H$ on the horizontal distribution $\H$.
\begin{definition}
    A distribution $\Lambda$ of the tangent bundle $TN$ is said to be \emph{bracket-generating} if it is generated by a finite number of Lie bracket of vector fields in $\Gamma(\Lambda)$.
\end{definition}
Whenever the horizontal subbundle $\H N$ of $TN$ is bracket-generating, the subellipticity of $\Delta_\H$ is guaranteed by H\"ormander's theorem. Moreover, \cite[Proposition 4.1.5]{BDW22} guarantees in that case its self-adjointness on $C_c^\infty(N)$, and its associated pre-Dirichlet form has a unique closed extension. On the other hand, as $\V$ is never bracket generating, we will not consider vertically lifted geodesic random walks.

We are now ready to state the invariance principle for the horizontal lift of the rescaled continuous-time random walk  for these types of Riemannian submersions. As a corollary, we obtain the associated relation between the Laplace-Beltrami operator and the horizontal Laplacian.
\begin{theorem}[Invariance principle for Riemannian submersions]\label{theorem:invariance_submersion}
    Let $(M,g)$ and $(N,\tilde g)$ be geodesically complete Riemannian manifolds equipped with a Riemannian submersion $\pi : N \to M$ with totally geodesic fibers such that the horizontal subbundle $\H N$ of $TN$ is bracket-generating. Let $\{\mu_p\}_{p\in M}$ be a distribution of increments on $M$ satisfying Assumption \ref{assumption:unit_variance} and Assumption \ref{assumption:third_moment}. 
    Let $\{\tilde{Z}^{(\alpha)}_t, t\geq 0\}$ be the process with generator \eqref{equation:resclaled_generator}.
    Then as $\alpha\to 0$, this process converges to horizontal Brownian motion; i.e., the process with generator $\frac{1}{2}\Delta_{\H}$.
\end{theorem}
\begin{corollary}\label{theorem:lift_of_laplacian}
    Let $\pi : N \to M$ be a Riemannian submersion with totally geodesic fibers, and assume the horizontal distribution $\H$ to be bracket-generating. Then the following identity hold:
    \begin{equation}\label{equation:horizontal_laplacian_equality_submersion}
        \Delta_\H\left(f \circ \pi\right) = \Delta_M f \circ \pi, \quad f \in C^\infty(M).
    \end{equation}
\end{corollary}
Before proving Theorem \ref{theorem:invariance_submersion}, let us go through some simple examples from \cite[Sections 4.1, 4.4]{BDW22} where the restrictions on the Riemannian submersion, namely, that the fibers are totally geodesic and that the horizontal distribution is bracket-generating, are verified.
\begin{itemize}
    \item The manifold $(M, g)$ itself, with $\Id_M : M \to M$ as submersion. The horizontal distribution is the whole tangent space. Theorem \ref{theorem:invariance_submersion} gives then a short proof of the invariance principle for geodesic random walks on Riemannian manifolds.
    \item The tangent bundle $\pi_{TM} : TM \to M$ equipped with the Sasaki metric \cite{Sa58} defined in terms of coordinates $(\{x_i\}_{1\le i\le d}, \{y_j\}_{1\le j\le d})$ about $(p, v)$ in $TM$ by
    \begin{equation*}
        ds^2 = g_{ij}dx^i dx^j + g_{ij}Dy^i Dy^j, 
    \end{equation*}
    where $D$ denotes the covariant differential with respect to $\overline{\nabla}$ on $M$; $Dy^k = dy^k + \overline{\Gamma}^k_{ij}y^i dx^j$.
    \item The orthonormal frame bundle $\pi_{O(M)} : O(M) \to M$ from Section 3.
    \item A general class of spaces on which such invariance principle holds are the principal bundles $\pi_P : P \to M$ with fiber Lie group $G$. Given a $G$-compatible connection form $\omega$ and a $G$-invariant metric $b$ on $G$, there is a unique metric $\tilde{g} = \pi^*g + b\omega$ on $P$ that makes $\pi_P$ into a Riemannian submersion with totally geodesic fibers such that the horizontal distribution $\H$ of $\omega$ is the orthogonal complement of the vertical distribution \cite[Theorem 3.5]{Vil70}. Whenever the horizontal distribution $\H$ is bracket-generating, the subellipticity of $\Delta_\H$ is guaranteed and there is a unique closed extension of its associated pre-Dirichlet form. The previous examples fall under this category.
\end{itemize}
The proof of Theorem \ref{theorem:invariance_submersion} is of course similar to the special case of Theorem \ref{theorem:invariance_frame_bundle} for the orthonormal frame bundle $O(M)$ presented in Section 3. Nevertheless, we do not have anymore the global canonical horizontal vector fields (\ref{equation:fundamental_horizontal}).
\begin{proof}[Proof of Theorem \ref{theorem:invariance_submersion}]
    Let $f : N \to \R$ be a smooth compactly supported function. We perform a Taylor
    s expansion of $f \circ \widetilde{\gamma_\alpha}$ around $u \in N_p$, where $\widetilde{\gamma_\alpha}$ is the horizontal lift starting at $u$ of the curve $\gamma_\alpha(t) = \exp_p(\alpha tv)$. There is some $0 < s < 1$ such that
    \begin{equation*}
        f \circ \widetilde{\gamma_\alpha}(1) = f(u) + \left.\frac{d}{dt}\right|_{t=0} \left(f \circ \widetilde{\gamma_\alpha}(t)\right) + \frac{1}{2} \left.\frac{d^2}{dt^2}\right|_{t=0} \left(f \circ \widetilde{\gamma_\alpha}(t)\right) + \frac{1}{6} \left.\frac{d^3}{dt^3}\right|_{t=s} \left(f \circ \widetilde{\gamma_\alpha}(t)\right).
    \end{equation*}
    The first time derivative is given by
    \begin{equation*}
        \frac{d}{dt} f(\widetilde{\gamma_\alpha}(t)) = df_{\widetilde{\gamma_\alpha}(t)} \circ \left(d\pi_{\widetilde{\gamma_\alpha}(t)}\right)^{-1}(\gamma_\alpha'(t)) = \left\langle \grad f, \widetilde{\gamma_\alpha}'(t) \right\rangle_{\tilde g} = \left\langle \left(\grad f\right)_\H, \widetilde{\gamma_\alpha}'(t) \right\rangle_{\tilde g},
    \end{equation*}
    where we used the fact that $\widetilde{\gamma_\alpha}$ is horizontal for the last equality. To obtain the second time derivative, we use the Levi-Civita connection $\nabla$ on $N$ and use the fact that $\widetilde{\gamma_\alpha}$ is a geodesic being the horizontal lift of a geodesic under a Riemannian submersion;
    \begin{align*}
        \frac{d^2}{dt^2} f(\widetilde{\gamma_\alpha}(t)) &= \left\langle \nabla_{\widetilde{\gamma_\alpha}'(t)} (\grad f)_\H, \widetilde{\gamma_\alpha}'(t)\right\rangle_{\tilde{g}} + \left\langle (\grad f)_\H, \nabla_{\widetilde{\gamma_\alpha}'(t)}  \widetilde{\gamma_\alpha}'(t)\right\rangle_{\tilde{g}} \\
        &= \left\langle \nabla_{\widetilde{\gamma_\alpha}'(t)} (\grad f)_\H, \widetilde{\gamma_\alpha}'(t)\right\rangle_{\tilde{g}}.
    \end{align*}
    In particular, at time $t = 0$, consider the orthonormal basis $\{E_i\}_{1\le i\le d}$ of $T_uN$ defined as the horizontal lift of an orthonormal basis $\{\overline{E}_i\}_{1\le i\le d}$ of $T_pM$. Write $v = v^i\overline{E}_i \in T_pM$. By linearity of the horizontal lift, we get $\widetilde{\gamma_\alpha}'(0) = \alpha v^iE_i$, and thus
    \begin{equation*}
        \left.\frac{d^2}{dt^2}\right|_{t=0} \left(f \circ \widetilde{\gamma_\alpha}(t)\right) = \alpha^2 v^iv^j \left\langle\nabla_{E_i} (\grad f)_\H, E_j\right\rangle_{\tilde{g}} = \alpha^2 v^iv^j \left(E_iE_j - \left(\nabla_{E_i} E_j\right)_\H\right) f.
    \end{equation*}
    By Assumption \ref{assumption:unit_variance} on the first and second moments, we deduce that
    \begin{equation*}
        \alpha^{-2}\int_{T_pM} \left.\frac{d}{dt}\right|_{t=0} \left(f \circ \widetilde{\gamma_\alpha}(t)\right) \: \mu_p(dv) = \alpha^{-1} df_u \circ \left(d\pi_u\right)^{-1} \left(\int_{T_pM}v \: \mu_p(dv)\right) = 0,
    \end{equation*}
    and
    \begin{align*}
        \alpha^{-2} \int_{T_pM} \left.\frac{d^2}{dt^2}\right|_{t=0} \left(f \circ \widetilde{\gamma_\alpha}(t)\right) \mu_p(dv) &= \int_{T_pM} v^iv^j \mu_p(dv) \left(E_iE_j - \left(\nabla_{E_i}E_j\right)_\H\right) f \\
        &=\sum_{i=1}^d \left(E_i^2 - \left(\nabla_{E_i}E_i\right)_\H\right) f.
    \end{align*}
    The last term is the horizontal Laplacian (\ref{equation:horizontal_laplacian_submersion}) for a submersion with totally geodesic fibers.

    For the third time derivative, first define the horizontal Hessian
    \begin{equation*}
        \Hess_\H f(Y, Z) = \left\langle \nabla_Y \left(\grad f\right)_\H, Z\right\rangle_{\tilde g}, \quad Y, Z \in \Gamma(TN),
    \end{equation*}
    which is a symmetric covariant tensor of order 2. Its covariant derivative is thus the tensor given by
    \begin{equation*}
        \nabla\Hess_\H f(X, Y, Z) = X\left(\Hess_\H f(Y, Z)\right) - \Hess_\H f(\nabla_X Y, Z) - \Hess_\H f(Y, \nabla_X Z), \quad X, Y, Z \in \Gamma(TN).
    \end{equation*}
    Note that, again since $\widetilde{\gamma_\alpha}$ is a geodesic,
    \begin{align*}
        \frac{d^3}{dt^3} f(\widetilde{\gamma_\alpha}(t)) &= \nabla_{\widetilde{\gamma_\alpha}'(t)} \left\langle \nabla_{\widetilde{\gamma_\alpha}'(t)}\left(\grad f\right)_\H, \widetilde{\gamma_\alpha}'(t)\right\rangle_{\tilde g} \\
        &= \nabla_{\widetilde{\gamma_\alpha}'(t)} \left\langle \nabla_{\widetilde{\gamma_\alpha}'(t)}\left(\grad f\right)_\H, \widetilde{\gamma_\alpha}'(t)\right\rangle_{\tilde g} - 2\left\langle \nabla_{\nabla_{\widetilde{\gamma_\alpha}'(t)} \widetilde{\gamma_\alpha}'(t)}\left(\grad f\right)_\H, \widetilde{\gamma_\alpha}'(t) \right\rangle_{\tilde g} \\
        &= \nabla \Hess_\H f\left(\widetilde{\gamma_\alpha}'(t), \widetilde{\gamma_\alpha}'(t), \widetilde{\gamma_\alpha}'(t)\right).
    \end{align*}
    Locally, $\nabla \Hess_\H f_u : T^3_uN \to \R$ is a bounded operator being linear on a finite dimensional vector space with operator norm given by
    \begin{equation*}
        C(u) = \max_{(\eta, \xi, \zeta) \in T^3_uN \: : \:\norm{\eta}_{\tilde g}\norm{\xi}_{\tilde g}\norm{\zeta}_{\tilde g} = 1} \left|\nabla \Hess_\H f_u(\eta, \xi, \zeta)\right|.
    \end{equation*}
    This constant $C(u)$ can be uniformly bounded since $|\nabla \Hess_\H f| : N \times T^3N \to \R$ is a continuous map on the compact set $\{(q,(\eta, \xi, \zeta)) \: : \: q \in \Supp f, (\eta, \xi, \zeta) \in T^3_qN, \norm{\eta}_{\tilde g}\norm{\xi}_{\tilde g}\norm{\zeta}_{\tilde g} = 1\}$, and hence attains a maximum $C > 0$. Since $\norm{\widetilde{\gamma_\alpha}'(t)}_{\tilde{g}} = \norm{\widetilde{\gamma_\alpha}'(0)}_{\tilde{g}} = \alpha\norm{v}_{g}$, we are able to conclude by Assumption \ref{assumption:third_moment} on the third moment;
    \begin{align*}
        \alpha^{-2}\int_{T_pM} \left.\frac{d^3}{dt^3}\right|_{t=s} \left(f \circ \widetilde{\gamma_\alpha}(t)\right) \mu_p(dv) &= \alpha^{-2} \int_{T_pM} \nabla \Hess_\H f\left(\widetilde{\gamma_\alpha}'(s), \widetilde{\gamma_\alpha}'(s), \widetilde{\gamma_\alpha}'(s)\right) \mu_p(dv) \\
        &\le \alpha C\sup_{q \in \pi\left(\Supp f\right)} \int_{T_qM} \norm{v}_g^3 \: \mu_q(dv),
    \end{align*}
    which goes to $0$ independently of $u$ as $\alpha \to 0$.
\end{proof}
\begin{proof}[Proof of Corollary \ref{theorem:lift_of_laplacian}]
    Consider an $\alpha$-rescaled continuous-time random walk on $M$ that satisfies Assumption \ref{assumption:unit_variance} and Assumption \ref{assumption:third_moment}. By Theorem \ref{theorem:invariance_submersion} with $\Id_M : M \to M$ as submersion, this process converges to Brownian motion; i.e., the process with generator $\frac{1}{2}\Delta_M$. On the other hand, by Theorem \ref{theorem:invariance_submersion}, the horizontal lift of the $\alpha$-rescaled random walk converges to horizontal Brownian motion; i.e., the process with generator $\frac{1}{2}\Delta_\H$ on $N$. Since both processes are Markov, and since the projection $\pi : N \to M$ is continuous, the corresponding generators must be $\pi$-related. This proves identity $(\ref{equation:horizontal_laplacian_equality_submersion})$ for the case of geodesically complete Riemannian manifolds and smooth compactly supported functions. The general case follows by restricting the diameter of the support of the collection of measures $\{\mu_p\}$ to, say, the unit ball, so that geodesics do not have arbitrarily large velocities. To extend beyond smooth compactly supported functions, a partition of unity argument concludes.
\end{proof}

\begin{remark}
Of course, Corollary \ref{theorem:fundamental_relation} is a special case of Corollary \ref{theorem:lift_of_laplacian}. Here, we proposed another approach than the classical one that we briefly outline for the sake of completeness. Essentially, the proof reduces to showing that the Levi-Civita connection $\nabla$ on $N$ is $\pi$-related to the Levi-Civita connection $\overline{\nabla}$ on $M$ (see \cite[Lemma 1]{ON66}). This follows from the fact that both the inner products for the specific metrics and the Lie brackets preserve $\pi$-relations;
\begin{equation*}
    \left\langle X, W \right\rangle_{\tilde g} = \left\langle \overline{X}, \overline{W}\right\rangle_g \circ \pi, \quad d\pi\left(\left[Y, Z\right]\right) = \left[\overline{Y}, \overline{Z}\right] \circ \pi
\end{equation*}
for $\pi$-related vector fields $W, X, Y, Z \in \Gamma(\H N)$ to $\overline{W}, \overline{X}, \overline{Y}, \overline{Z} \in \Gamma(TM)$, and hence
\begin{equation*}
    \left\langle X, [Y, Z] \right\rangle_{\tilde g} = \left\langle \overline{X}, [\overline{Y}, \overline{Z}] \right\rangle_g \circ \pi.
\end{equation*}
It remains to express the Levi-Civita connection $\nabla$ on $N$ via Koszul's formula for any triple $X, Y, Z \in \Gamma(TN)$;
\begin{equation*}
    2\left\langle \nabla_X Y,Z\right\rangle_{\tilde g} = X\langle Y,Z\rangle_{\tilde g}  + Y\langle X, Z\rangle_{\tilde g} - Z\langle X,Y\rangle_{\tilde g} -\left\langle X, [Y,Z]\right\rangle_{\tilde g} -\left\langle Y, [X, Z]\right\rangle_{\tilde g} + \left \langle Z, [X, Y]\right\rangle_{\tilde g} .
\end{equation*}
\end{remark}

{\bf Acknowledgment:} This publication is part of the project Interacting particle systems and Riemannian geometry (with project number OCENW.M20.251) of the research program Open Competitie ENW which is (partly) financed by the Dutch Research Council (NWO)\footnote{\includegraphics[height=2cm, width=5cm]{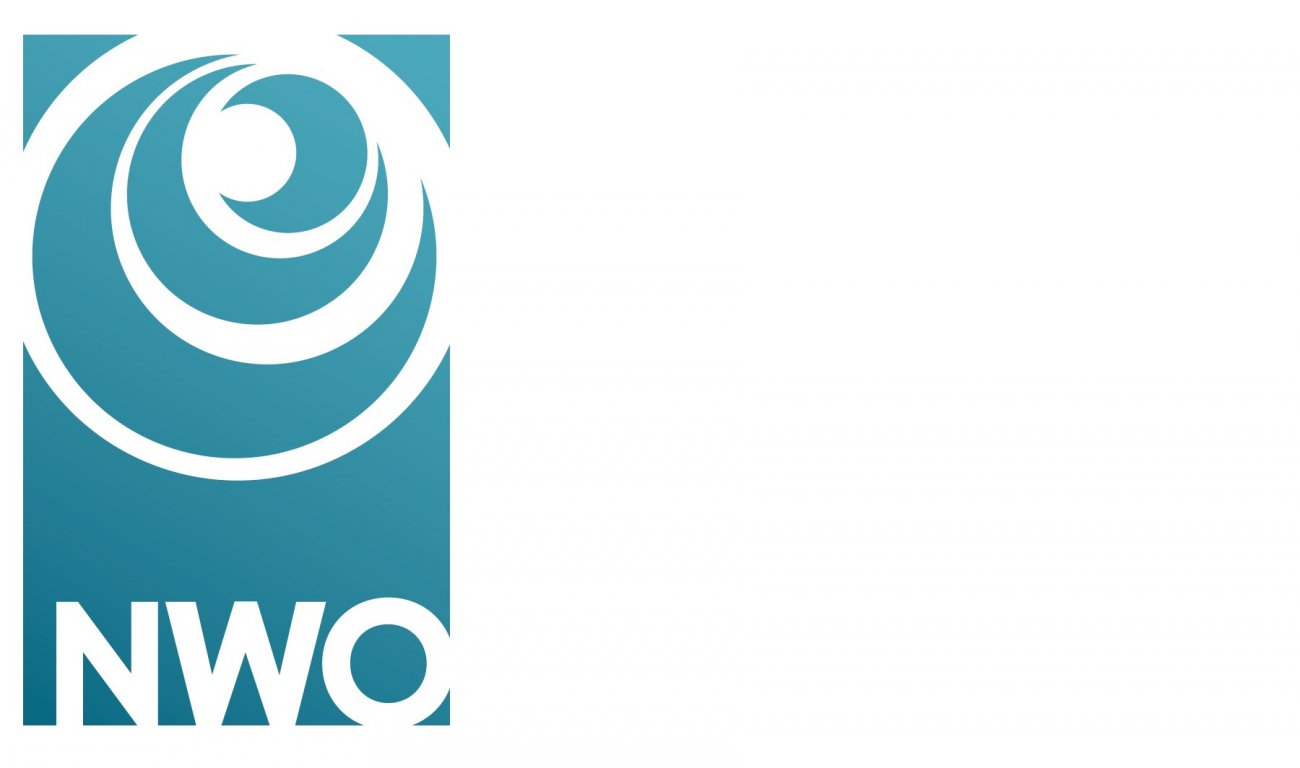}}.

\printbibliography
\end{document}